\documentclass{birkjour}
 \newtheorem{thm}{Theorem}[section]
 \newtheorem{cor}[thm]{Corollary}
 \newtheorem{lem}[thm]{Lemma}
 \newtheorem{prop}[thm]{Proposition}
 \theoremstyle{definition}
 
 \theoremstyle{remark}

 \numberwithin{equation}{section}

\begin{document}

%
%
%
%
%
%
%
%
%

\title[Reproducing kernel]
 {Reproducing kernel of the space $R^t(K,\mu)$}

\author[L. Yang]{Liming Yang}

\address{%
Department of Mathematics\\
Virginia Polytechnic Institute and State University\\
Blacksburg, VA 24061\\
USA}

\email{yliming@vt.edu}




\dedicatory{Dedicated to the memory of Ronald G. Douglas}

\begin{abstract}
For $1 \le t < \infty ,$ a compact subset $K$ of the complex plane $\mathbb C,$ and a finite positive measure $\mu$ supported on $K,$ $R^t(K, \mu)$ denotes the closure in $L^t (\mu )$ of rational functions with poles off $K$. Let $\Omega$ be a connected component of the set of analytic bounded point evaluations for $R^t(K, \mu)$. In this paper, we examine the behavior of the reproducing kernel of $R^t(K, \mu)$ near the boundary $\partial \Omega \cap \mathbb T$, assuming that $\mu (\partial \Omega \cap \mathbb T ) > 0$,  where $\mathbb T$ is the unit circle.
\end{abstract}

\maketitle

\section{Introduction}

Throughout this paper, let $\mathbb{D}$ denote the unit
disk $\{z: |z| < 1\}$ in the complex plane $\mathbb{C}$, let
$\mathbb{T}$ denote the unit circle $\{z: |z| =1\}$, let
$m$ denote normalized Lebesgue measure on $\mathbb{T}$. Let $\mu$ be a finite, positive Borel measure that is compactly 
supported in $\mathbb{C}$. We require that the support of $\mu$ be contained in
some compact set $K$ and we indicate this by $\mbox{spt}(\mu)~\subseteq~K$. Under these circumstances
and for $1\leq t < \infty$ and $t' = \frac{t}{t-1}$, functions in $\mathcal P$(the set of analytic polynomials) and $\mbox{Rat}(K) := \{q:\mbox{$q$ is a rational function 
with poles off $K$}\}$ are members of $L^t(\mu)$. We let $P^t(\mu)$ denote the closure of $\mathcal P$ in $L^t(\mu)$ and let $R^t(K, \mu)$ denote the closure of $\mbox{Rat}(K)$ in $L^t(\mu)$.
A point $z_0$ in $\mathbb{C}$ is called a \textit{bounded point evaluation} for $P^t(\mu)$ (resp., $R^t(K, \mu)$)
if $f\mapsto f(z_0)$ defines a bounded linear functional for functions in $\mathcal P$ (resp., $\mbox{Rat}(K)$)
with respect to the $L^t(\mu)$ norm. The norm of the bounded linear functional is denoted by $M_{z_0}$.
The collection of all such points is denoted $\mbox{bpe}(P^t(\mu))$ 
(resp., $\mbox{bpe}(R^t(K, \mu)$)).  If $z_0$ is in the interior of $\mbox{bpe}(P^t(\mu))$ (resp., $\mbox{bpe}(R^t(K, \mu)$)) 
and there exist positive constants $r$ and $M$ such that $|f(z)| \leq M\|f\|_{L^t(\mu)}$, whenever $|z - z_0|\leq r$ 
and $f\in \mathcal P$ (resp., $f\in \mbox{Rat}(K)$), then we say that $z_0$ is an 
\textit{analytic bounded point evaluation} for $P^t(\mu)$ (resp., $R^t(K, \mu)$). The collection of all such 
points is denoted $\mbox{abpe}(P^t(\mu))$ (resp., $\mbox{abpe}(R^t(K, \mu)$)). Actually, it follows from Thomson's Theorem
\cite{thomson} (or see Theorem \ref{TThomson}, below) that $\mbox{abpe}(P^t(\mu))$ is the interior of $\mbox{bpe}(P^t(\mu))$. 
This also holds in the context of $R^t(K, \mu)$ as was shown by J. Conway and N. Elias in \cite{ce93}. Now, 
$\mbox{abpe}(P^t(\mu))$ is the largest open subset of $\mathbb{C}$ to which every function in $P^t(\mu)$ has an analytic 
continuation under these point evaluation functionals, and similarly in the context of $R^t(K, \mu)$. Let $S_\mu$ denote the multiplication by $z$ on $R^t(K, \mu)$. It is well known that $R^t(K, \mu) = R^t(\sigma(S_\mu), \mu)$ and $\sigma(S_\mu)\subset K$, where $\sigma(S_\mu)$ denotes the spectrum of $S_\mu$ (see, for example, Proposition 1.1 in \cite{ce93}). Throughout this paper, we assume $K = \sigma(S_\mu)$.

Our story begins with celebrated results of J. Thomson, in \cite{thomson}. 
\medskip

\begin{thm}[Thomson (1991)]\label{TThomson} 
Let $\mu$ be a finite, positive Borel measure that is compactly supported in $\mathbb{C}$ and suppose that $1\leq t < \infty$.
There is a Borel partition $\{\Delta_i\}_{i=0}^\infty$ of $\mbox{spt}(\mu)$ such that the space $P^t(\mu |_{\Delta_i})$ 
contains no nontrivial characteristic function (i.e., 
 $P^t(\mu |_{\Delta_i})$ is irreducible) and
 \[
 \ P^t(\mu ) = L^t(\mu |_{\Delta_0})\oplus \left \{ \oplus _{i = 1}^\infty P^t(\mu |_{\Delta_i}) \right \}.
 \]
Furthermore, if $U_i :=abpe( P^t(\mu |_{\Delta_i}))$ for $i \ge 1,$ then $U_i$ is a simply connected region and $\Delta_i\subseteq \overline{U_i}$.
\end{thm}
\medskip

We mention a remarkable result of A. Aleman, S. Richter and C. Sunberg. It's proof involves a modification of Thomson's scheme along
with results of X. Tolsa on analytic capacity. 
\medskip

\begin{thm}[Aleman, Richter and Sunberg (2009)] \label{ARSTheorem}
Suppose that $\mu$ is supported in $\overline{\mathbb D}$, 
$abpe (P^t(\mu )) = \mathbb{D}$, $P^t(\mu )$ is irreducible, and that $\mu (\mathbb{T})> 0$.
\newline
(a) If $f \in P^t(\mu )$, then the nontangential limit $f^*(\zeta )$ of f at $\zeta$ exists a.e. $\mu |_{\mathbb{T}}$ 
and $f^* = f |_{\mathbb{T}}$ as elements of $L^t(\mu |_{\mathbb{T}}).$
\newline
(b) Every nontrivial, closed invariant subspace $\mathcal{M}$ for the shift $S_{\mu}$ on $P^t(\mu )$ has index 1; that is, the dimension
of $\mathcal{M}/z\mathcal{M}$ is one.
\newline
(c) If $t > 1$, then
\[
 \ \lim_{\lambda\rightarrow z} (1 - |\lambda |^2) ^{\frac{1}{t}} M_{\lambda } = \dfrac{1}{h(z)^{\frac{1}{t}}}
 \]
nontangentially for $m$-a.a. $z \in \mathbb T$, where $\mu |_{\mathbb T} = hm$.
\end{thm}
\medskip

J. Thomson's proof of the existence of bounded point evaluations for $P^t(\mu)$
uses Davie's deep estimation of analytic capacity, S. Brown's technique, and Vitushkin's localization for 
uniform rational approximation. The proof is excellent but complicated, and it does not really
lend itself to showing the existence of nontangential boundary values in the case 
that $\mbox{spt}(\mu)\subseteq\overline{\mathbb{D}}$, $P^t(\mu)$ is irreducible and $\mu(\mathbb{T}) > 0$.
X. Tolsa's remarkable results on analytic capacity opened the door for a new view of things, through the works
of \cite{acy18}, \cite{ars09}, \cite{ARS10} and \cite{b08}, etc. 

In this paper, we assume that $R^t(K,\mu )$ is irreducible and $\Omega$ is a connected region satisfying:
 \begin{eqnarray}\label{ConditionsOnK}
 \ abpe(R^t(K,\mu )) = \Omega ,~ K = \overline \Omega , ~ \Omega \subset \mathbb D,~ \mathbb T \subset \partial\Omega .
 \end{eqnarray}
It is well known that, in this case, $\mu |_{\mathbb T} << m. $
 So we assume $\mu |_{\mathbb T} = hm.$ 

For $ \delta > 0$ and $\lambda\in\mathbb C$, set $B(\lambda, \delta) = \{z: ~ |z - \lambda | < \delta\}$. 
For $0 < \sigma < 1$, let $\Gamma _\sigma (e^{i\theta})$ denote the polynomial convex hull of $\{e^{i\theta}\}$ and $B(0,\sigma)$.
Define $\Gamma ^\delta _\sigma (e^{i\theta}) = \Gamma _\sigma (e^{i\theta})\cap B(e^{i\theta}, \delta).$ In order to define a nontangential limit of a function in $R^t(K,\mu )$ at $e^{i\theta} \in \partial \Omega,$ one needs $\Gamma ^\delta _\sigma (e^{i\theta}) \subset \Omega$ for some $\delta.$ Therefore, we define the strong outer boundary of $\Omega$ as the following:
 \begin{eqnarray}\label{SOBoundary}
 \ \partial _{so, \sigma} \Omega = \{e^{i\theta} \in \partial \Omega: ~\exists 0<\delta<1,~ \Gamma _{\sigma}^\delta (e^{i\theta}) \subset \Omega \}. 
 \end{eqnarray}
It is known that $\partial _{so, \sigma} \Omega$ is a Borel set (i.e., see Lemma 4 in \cite{ot80}) and $m(\partial _{so, \sigma_1} \Omega \setminus \partial _{so, \sigma _2} \Omega ) = 0$ for $\sigma _1 \ne \sigma _2.$ Therefore, we set $\partial _{so} \Omega = \partial _{so, \frac{1}{2}} \Omega$.

The paper \cite{acy18} presents an alternate and simpler route to prove Theorem \ref{ARSTheorem} (a) and (b)  
that has extension to the context of mean rational approximation as in Theorem \ref{ACYTheorem} below. It also uses the results of X. Tolsa on analytic capacity.
\medskip
 
\begin{thm}[Akeroyd, Conway and Yang (2019)]\label{ACYTheorem}
Let $\Omega$ be a bounded connected open set satisfying \eqref{ConditionsOnK}.
Suppose that $\mu$ is a finite positive measure supported in $K$, $abpe(R^t(K,\mu )) = \Omega $, $R^t(K,\mu )$ is irreducible, $\mu |_{\mathbb T} = hm,$ and $\mu (\partial _{so} \Omega ) > 0.$ Then:
\newline
(a) If $f \in R^t(K,\mu )$ then the nontangential limit $f^*(z)$ of $f$ exists for $\mu |_{\partial _{so} \Omega}$ 
almost all z, and $f^* = f |_{\partial _{so} \Omega}$ as elements of $L^t(\mu |_{\partial _{so} \Omega}).$
\newline
(b) Every nonzero rationally invariant subspace $\mathcal M$ of $R^t(K,\mu )$ has index 1, that is, $dim(\mathcal M / (S_\mu - \lambda _0) \mathcal M) = 1,$ for $\lambda _0\in \Omega.$
\end{thm}
\medskip

Theorem \ref{ACYTheorem} is a direct application of Theorem 3.6 in \cite{acy18}, which proves a generalized Plemelj's formula for a compactly supported finite complex-valued measure. In fact, the generalized Plemelj's formula holds for rectifiable curve (other than $\mathbb T$), so Theorem \ref{ACYTheorem} is valid if $\partial \Omega$ is a certain rectifiable curve. 

In this paper, we continue the work of section 3 in \cite{acy18} to generalize Theorem \ref{ARSTheorem} (c). We refine the estimates of Cauchy transform of a finite measure in \cite{acy18} and provide an alternate proof of Theorem \ref{ARSTheorem} (c) that can extend the result to the context of certain mean rational approximation space $R^t(K, \mu)$. 

By Riesz representation theorem, there exists $k_\lambda\in L^{t'}(\mu)$ for $\lambda\in abpe(R^t(K,\mu ))$ such that $M_\lambda = \|k_\lambda \|_{L^{t'}(\mu)}$ and
 \[
 \ f(\lambda) = \int f(z) \bar k_\lambda (z) d\mu(z),~ f \in \mbox{Rat}(K).
 \]
The function $k_\lambda$ is called the reproducing kernel for $R^t(K,\mu )$. 
\medskip
 
\begin{thm}[Main Theorem]\label{MTheorem}
Let $\Omega$ be a bounded connected open set satisfying \eqref{ConditionsOnK}.
Suppose that $\mu$ is a finite positive measure supported in $K$, $abpe(R^t(K,\mu )) = \Omega $, $R^t(K,\mu )$ is irreducible, $\mu |_{\mathbb T} = hm,$ and $\mu (\partial _{so} \Omega ) > 0.$ If $t > 1,$ then 
 \[
 \ \lim_{ \Gamma _{\frac{1}{4}}(e^{i\theta})\ni\lambda \rightarrow e^{i\theta}} (1 - |\lambda |^2)^{\frac{1}{t}} M_\lambda = \lim_{ \Gamma _{\frac{1}{4}}(e^{i\theta})\ni\lambda \rightarrow e^{i\theta}} (1 - |\lambda |^2)^{\frac{1}{t}} \|k_\lambda \|_{L^{t'}(\mu)} = \dfrac{1}{h(e^{i\theta})^{\frac{1}{t}}}
 \]
for $\mu$-almost all $e^{i\theta}\in \partial _{so} \Omega.$
\end{thm}
\bigskip

\section{Proof of Main Theorem}

Let $\nu$ be a finite complex-valued Borel measure that
is compactly supported in $\mathbb {C}$. 
For $\epsilon > 0,$ $\mathcal C_\epsilon(\nu)$ is defined by
\begin{eqnarray}\label{CTEDefinition}
\ \mathcal C_\epsilon(\nu)(z) = \int _{|w-z| > \epsilon}\dfrac{1}{w - z} d\nu (w).
\end{eqnarray} 
The (principal value) Cauchy transform
of $\nu$ is defined by
\begin{eqnarray}\label{CTDefinition}
\ \mathcal C(\nu)(z) = \lim_{\epsilon \rightarrow 0} \mathcal C_\epsilon(\nu)(z)
\end{eqnarray}
for all $z\in\mathbb{C}$ for which the limit exists.
If $\lambda \in \mathbb{C}$ and $\int \frac{d|\nu |}{|z - \lambda |} < \infty$, then 
$\lim_{r\rightarrow 0}\frac{|\nu |(B(\lambda, r))}{r} = 0$ and 
$\lim_{\epsilon \rightarrow 0} \mathcal C_{\epsilon}(\nu)(\lambda )$ exists. Therefore, a standard application of Fubini's
Theorem shows that $\mathcal C(\nu) \in L^s_{\mbox{loc}}(\mathbb{C})$, for $ 0 < s < 2$. In particular, it is
defined for almost all $z$ with respect to area measure on $\mathbb{\mathbb{C}}$, and clearly $\mathcal{C}(\nu)$ is analytic 
in $\mathbb{C}_\infty \setminus\mbox{spt}(\nu)$, where $\mathbb{C}_\infty := \mathbb{C} \cup \{\infty \}.$ In fact, from 
Corollary 3.1 in \cite{acy18}, we see that \eqref{CTDefinition} is defined for all $z$ except for a set of zero analytic 
capacity. Thoughout this section, the Cauchy transform of a measure always means the principal value of the transform.

The maximal Cauchy transform is defined by
 \[
 \ \mathcal C_*(\nu)(z) = \sup _{\epsilon > 0}| \mathcal C_\epsilon(\nu)(z) |.
 \]

If $K \subset\mathbb{C}$ is a compact subset, then  we
define the analytic capacity of $K$ by
\[
\ \gamma(K) = \sup |f'(\infty)|,
\]
where the supremum is taken over all those functions $f$ that are analytic in $\mathbb C_{\infty} \setminus K$ such that
$|f(z)| \le 1$ for all $z \in \mathbb{C}_\infty \setminus K$; and
$f'(\infty) := \lim _{z \rightarrow \infty} z(f(z) - f(\infty)).$
The analytic capacity of a general subset $E$ of $\mathbb{C}$ is given by: 
\[
\ \gamma (E) = \sup \{\gamma (K) : K\stackrel{\mbox{\tiny{$\subset\subset$}}}{} E\}.
\]
Good sources for basic information about analytic
capacity are Chapter VIII of \cite{gamelin}, Chapter V of \cite{conway}, and \cite{Tol14}.

A related capacity, $\gamma _+,$ is defined for subsets $E$ of $\mathbb{C}$ by:
\[
\ \gamma_+(E) = \sup \|\mu \|,
\]
where the supremum is taken over positive measures $\mu$ with compact support
contained in $E$ for which $\|\mathcal{C}(\mu) \|_{L^\infty (\mathbb{C})} \le 1.$ 
Since $\mathcal C\mu$ is analytic in $\mathbb{C}_\infty \setminus \mbox{spt}(\mu)$ and $(\mathcal{C}(\mu)'(\infty) = \|\mu \|$, 
we have:
\[
\ \gamma _+(E) \le \gamma (E)
\]
for all subsets $E$ of $\mathbb{C}$.  X. Tolsa has established the following astounding results.

\begin{thm}[Tolsa (2003)]\label{TTolsa}

(1) $\gamma_+$ and $\gamma$ are actually equivalent. 
That is, there is an absolute constant $A_T$ such that 
\begin{eqnarray}\label{GammaEq}
\ \gamma (E) \le A_ T \gamma_+(E)
\end{eqnarray}
for all $E \subset \mathbb{C}.$ 

(2) Semiadditivity of analytic capacity:
\begin{eqnarray}\label{Semiadditive}
\ \gamma \left (\bigcup_{i = 1}^m E_i \right ) \le A_T \sum_{i=1}^m \gamma(E_i)
\end{eqnarray}
where $E_1,E_2,...,E_m \subset \mathbb{C}.$

(3) There is an absolute positive constant $C_T$ such that, for any $a > 0$, we have:  
\begin{eqnarray}\label{WeakOneOne}
\ \gamma(\{\mathcal{C}_*(\nu)  \geq a\}) \le \dfrac{C_T}{a} \|\nu \|.
\end{eqnarray}
\end{thm}

\begin{proof}
(1) and (2) are from \cite{Tol03} (also see Theorem 6.1 and Corollary 6.3 in \cite{Tol14}).

(3) follows from Proposition 2.1 of \cite{Tol02} (also see \cite{Tol14} Proposition 4.16).
\end{proof}
\medskip

The following lemma is a modification of Lemma 3.2 of \cite{acy18}. 

\begin{lem}\label{CauchyTLemma} 
Let $\nu$ be a finite  measure supported in $\bar {\mathbb D}$ and $| \nu | (\mathbb T ) = 0.$ Let $1 < p \le \infty ,$ $q = \frac{p}{p-1},$ $f \in C (\bar {\mathbb D}),$ and $g \in L^{q} (| \nu |).$ Assume that for some $e^{i\theta} \in \mathbb T$ we have:
\begin{eqnarray}\label{zerodensity}
\ \lim_{r\rightarrow 0} \dfrac{\int _{B(e^{i\theta}, r)}|g|^qd| \nu |}{r } = 0
\end{eqnarray}
Then, for any $a > 0$, there exists $\delta_a$, $0 < \delta_a < \frac{1}{4}$, such that whenever $0 < \delta < \delta_a$,  
there is a subset $E_{\delta}^f$ of $B(e^{i\theta}, \delta)$ 
and $\epsilon (\delta ) > 0$ satisfying: 
 \ \begin{eqnarray}\label{CT1}
 \ \lim _{\delta \rightarrow 0} \epsilon(\delta ) = 0,
 \ \end{eqnarray} 
 \ \begin{eqnarray}\label{CT2}
 \ \gamma(E_\delta ^f) <\epsilon (\delta ) \delta ,
 \ \end{eqnarray}
for all $\lambda\in B (e^{i\theta}, \delta ) \setminus E_\delta^f,$ $|\lambda _0 -  e^{i\theta} | = \frac{\delta}{2} $ and $\lambda _0\in \Gamma _\frac{1}{2}(e^{i\theta })$,
 \begin{eqnarray}\label{CT30}
 \ \lim_{\epsilon \rightarrow 0}\mathcal{C} _{\epsilon}\left ((1 - \bar \lambda _0 z)^{\frac{2}{p}}\delta^{-\frac{1}{p}}fg\nu \right )(\lambda)
 \end{eqnarray} 
exists,  and 
\ \begin{eqnarray}\label{CT3}
\ \begin{aligned}
\ & \left  |\mathcal C\left ((1 - \bar \lambda _0 z)^{\frac{2}{p}}\delta^{-\frac{1}{p}}fg\nu \right )(\lambda) - \mathcal C\left ((1 - \bar \lambda _0 z)^{\frac{2}{p}}\delta^{-\frac{1}{p}}fg\nu \right )(\frac{1}{\bar{\lambda} _0}) \right |\\
\ \le & a\|f\|_{L^{p} (| \nu |)}.
\ \end{aligned}
\ \end{eqnarray}
Notice that the set $E_\delta ^f$ depends on $f$ and all other parameters are independent of $f.$
\end{lem}

\begin{proof} Let 
 \[ 
 \ M = \sup _{r > 0 }\frac{\int _{B(e^{i\theta}, r)}|g|^qd| \nu |}{r }.
 \]
 Then, by \eqref{zerodensity}, $M < \infty.$
For $a > 0$, choose $N$ and $\delta_a$, $0 < \delta_a < \frac{1}{4}$, satisfying:
 \[
 \ N = 6 + \left (\dfrac{256}{a} \sum_{k=0}^\infty 2^{\frac{-k}{q}}  \right)^q M,
 \]
 \[
 \ \left( \dfrac{\int _{B(\lambda _0, N\delta)}|g|^qd| \nu |}{\delta }\right)^\frac{1}{q} < \dfrac{a}{4^{3+\frac{2}{q}}}
 \]
for $0 < \delta < \delta_a.$
We now fix $\delta$, $0 < \delta < \delta_a$, and let 
 \[ 
 \ \nu_\delta = \dfrac{\chi _{B (e^{i\theta}, N \delta )}}{(1 - \bar{\lambda} _0 z )^{1-\frac{2}{p}}\delta^{\frac{1}{p}}}fg \nu ,
\]
 where $\chi _A$ denotes the characteristic function of the set $A$.  
For $0 < \epsilon < \delta $ and $\lambda \in B (e^{i\theta} , \delta ),$ we get:
 \[
 \ 2(1 - |\lambda _0|) \le \delta \le 4(1 - |\lambda _0|),
 \]
 \[
 \ \overline {B (\lambda , \epsilon)} \subset B (e^{i\theta}, 2 \delta ) \subset B (e^{i\theta}, N\delta),
 \]
and
 \[ 
 \ \begin{aligned}
 \ & \left  |\mathcal C _\epsilon \left ((1 - \bar \lambda _0 z)^{\frac{2}{p}}\delta^{-\frac{1}{p}}fg\nu \right )(\lambda) - \mathcal C  \left ((1 - \bar \lambda _0 z)^{\frac{2}{p}}\delta^{-\frac{1}{p}}fg\nu \right )(\frac{1}{\bar{\lambda} _0}) \right | \\
 \ \le & \dfrac{|1 - \bar{\lambda} _0 \lambda |}{\delta^{\frac{1}{p}}} \left | \int _{|z - \lambda| > \epsilon}\dfrac{fgd\nu  }{(z - \lambda)(1 - \bar{\lambda}_0 z)^{1-\frac{2}{p}}} \right | \\
\ & + \left | \mathcal C  \left (\chi _{\bar B (\lambda, \epsilon)}\dfrac{(1 - \bar \lambda _0 z)^{\frac{2}{p}}}{\delta^{\frac{1}{p}}}fg\nu \right )(\frac{1}{\bar{\lambda} _0}) \right |\\
\ \le & 2\delta^{\frac{1}{q}}  \left | \int _{ B (e^{i\theta}, N\delta )^c}\dfrac{fgd\nu}{(z - \lambda)(1 - \bar{\lambda}_0 z)^{1-\frac{2}{p}}} \right | + 2\delta \left |\int _{|z - \lambda| > \epsilon}\dfrac{d\nu_\delta}{(z - \lambda)} \right | \\
 \ & + \int_{\bar B (\lambda, \epsilon)}\dfrac{\delta^{-\frac{1}{p}}}{|1 - \bar \lambda _0 z|^{1-\frac{2}{p}}}|fg|d|\nu | \\
 \ \le & 2\delta^{\frac{1}{q}} \sum_{k=0}^{\infty}\int _{2^kN\delta \le |z - e^{i\theta} | < 2^{k+1}N\delta} \dfrac{|f||g|d|\nu |}{|z - \lambda ||1 - \bar{\lambda} _0 z |^{1-\frac{2}{p}} }  + 2\delta |\mathcal C_\epsilon \nu _\delta (\lambda )| \\
 \ & + \int_{ B (e^{i\theta}, 2 \delta)}\dfrac{|1 - \bar \lambda _0 z|\delta^{-\frac{1}{p}}}{|1 - \bar \lambda _0 z|^{\frac{2}{q}}}|fg|d|\nu | \\
 \ \le & 2\delta^{\frac{1}{q}} \sum_{k=0}^{\infty} \dfrac{(2^{k+1}N\delta)^{\frac{1}{q}}(2^kN\delta  + 2\delta )^{\frac{2}{p}}}{(2^kN\delta  - \delta )(2^kN\delta  - 2\delta )} \left (\dfrac{\int _{B (e^{i\theta}, 2^{k+1}N\delta)}|g|^qd|\nu |}{2^{k+1}N\delta} \right )^{\frac{1}{q}} \|f\|_{L^{p} (| \nu |)} \\
\ &+ 2\delta \mathcal  C_* \nu _\delta (\lambda ) + 4\int_{ B (e^{i\theta}, 2 \delta)}\dfrac{\delta^{\frac{1}{q}}}{|1 - \bar \lambda _0 z|^{\frac{2}{q}}}|fg|d|\nu |\\
 \ \le & \dfrac{4(N+2)^{1+\frac{1}{p}}\sum_{k=0}^\infty 2^{\frac{-k}{q}}M^{\frac{1}{q}}}{(N-1)(N-2)} \|f\|_{L^{p} (| \nu |)} + 2\delta \mathcal C_* \nu _\delta (\lambda ) \\
\ & + 4^{1+\frac{2}{q}} \|f\|_{L^{p} (| \nu |)} \left (\dfrac{\int_{ B (e^{i\theta}, 2 \delta)} |g|^qd|\nu |}{\delta} \right )^{\frac{1}{q}}\\
\ \le &\dfrac{a}{4} \|f\|_{L^{p} (| \nu |)} + 2\delta  \mathcal C_* \nu _\delta (\lambda ).
\ \end{aligned}
 \]
Let
 \[
 \ \mathcal{E}_{\delta} = \{\lambda : C_*(\nu _{\delta})(\lambda ) \ge \frac{a\|f\|_{L^{p} (| \nu |)}}{8\delta} \} \cap B (\lambda _0, \delta).
 \]
Then
 \[
 \ \begin{aligned}
 \ & \{\lambda : |\mathcal C _\epsilon \left ((1 - \bar \lambda _0 z)^{\frac{2}{p}}\delta^{-\frac{1}{p}}fg\nu \right )(\lambda) - \mathcal C  \left ((1 - \bar \lambda _0 z)^{\frac{2}{p}}\delta^{-\frac{1}{p}}fg\nu \right )(\frac{1}{\bar{\lambda} _0})| \\
\ &\ge a \|f\|_{L^{p} (| \nu |)}\} \cap B (\lambda _0, \delta) \subset \mathcal{E}_{\delta}.
 \ \end{aligned}
 \]
From Theorem \ref{TTolsa} (3), we get
 \[
 \ \gamma (\mathcal{E}_{\delta}) \le \dfrac{8C_T\delta}{a\|f\|_{L^{p} (| \nu |)}} \| \nu _{\delta} \| \le \dfrac{32C_T\delta}{a}\left (\dfrac{\int_{ B (e^{i\theta}, N \delta)} |g|^qd|\nu |}{\delta} \right )^{\frac{1}{q}} .
 \]
Let $E$ be the set of $\lambda\in \mathbb{C}$ such that $\lim_{\epsilon\rightarrow 0}\mathcal C _\epsilon \left (fg\nu \right )(\lambda)$ does not exist. By Corollary 3.1 in \cite{acy18}, 
we see that $\gamma (E) = 0$. It is clear that \eqref{CT30} exists for $\lambda \in \mathbb D \setminus E$ because
 \[
 \ \begin{aligned}
 \ & \lim_{\epsilon\rightarrow 0}\mathcal C _\epsilon \left ((1 - \bar \lambda _0 z)^{\frac{2}{p}}\delta^{-\frac{1}{p}}fg\nu \right )(\lambda) - (1 - \bar \lambda _0 \lambda )^{\frac{2}{p}}\delta^{-\frac{1}{p}}\lim_{\epsilon\rightarrow 0}\mathcal C _\epsilon \left (fg\nu \right )(\lambda) \\
\ = & \int \dfrac{(1 - \bar \lambda _0 z)^{\frac{2}{p}}\delta^{-\frac{1}{p}} - (1 - \bar \lambda _0 \lambda )^{\frac{2}{p}}\delta^{-\frac{1}{p}}}{z - \lambda}fgd\nu
 \ \end{aligned}
 \]
exists for all $\lambda \in \mathbb D \setminus E$. 

Now let $E_{\delta}^f = \mathcal{E}_{\delta} \cup E$. Applying Theorem \ref{TTolsa} (2) we find that
 \[
 \ \gamma(E_{\delta}) \le A_T (\gamma (\mathcal{E}_{\delta})+ \gamma (E)) < \dfrac{32A_TC_T}{a}\left (\dfrac{\int_{ B (e^{i\theta}, N \delta)} |g|^qd|\nu |}{\delta} \right )^{\frac{1}{q}}  \delta. 
 \]
Letting
 \[
 \ \epsilon (\delta) = \dfrac{32A_TC_T}{a}\left (\dfrac{\int_{ B (e^{i\theta}, N \delta)} |g|^qd|\nu |}{\delta} \right )^{\frac{1}{q}}, 
 \]
we conclude that \eqref{CT1} and \eqref{CT2} hold.
On $B (\lambda _0, \delta ) \setminus E_\delta$ and for $\epsilon < \delta$, we conclude that
 \[
 \ \left  |\mathcal C _\epsilon \left ((1 - \bar \lambda _0 z)^{\frac{2}{p}}\delta^{-\frac{1}{p}}fg\nu \right )(\lambda) - \mathcal C  \left ((1 - \bar \lambda _0 z)^{\frac{2}{p}}\delta^{-\frac{1}{p}}fg\nu \right )(\frac{1}{\bar{\lambda} _0}) \right | < a \|f\|_{L^{p} (| \nu |)}.
 \]
Therefore, \eqref{CT3} follows since
 \[
 \ \lim_{_\epsilon\rightarrow 0} \mathcal C _\epsilon \left ((1 - \bar \lambda _0 z)^{\frac{2}{p}}\delta^{-\frac{1}{p}}fg\nu \right )(\lambda) = \mathcal C \left ((1 - \bar \lambda _0 z)^{\frac{2}{p}}\delta^{-\frac{1}{p}}fg\nu \right )(\lambda).
 \]
\end{proof}
\medskip

\begin{prop}\label{MProposition2} Let $\mu$ be a finite positive  measure with support in $ K \subset\bar {\mathbb D}$  and $\mu | _{\mathbb T} = hm.$  Let $1 < p <\infty, ~ q = \frac{p}{p-1}, ~ f\in C(\bar{\mathbb D}), ~ g \in L^q (\mu ),$ and $fg\mu \perp Rat(K).$ Then for $0 < \beta < \frac{1}{16},$ $b > 0,$ and $m$-almost all $e^{i\theta}\in \mathbb T,$ there exist $0 < \delta_a < \frac{1}{4},$ $E_{\delta}^f \subset B(e^{i\theta}, \delta),$ and $\epsilon (\delta ) > 0,$ where $0 < \delta < \delta_a,$ such that $\lim_{\delta\rightarrow 0}\epsilon (\delta ) = 0,$ $\gamma(E_\delta ^f) < \epsilon (\delta ) \delta ,$ and for $\lambda _0 \in (\partial B (e^{i\theta}, \frac{\delta}{2} )) \cap \Gamma _{\frac{1}{4}}(e^{i\theta}),$
\[
\ \left  |\mathcal C \left (\dfrac{(1 - \bar \lambda _0 z)^{\frac{2}{p}}}{ (1-|\lambda _0|^2)^{\frac{1}{p}}}fg\mu \right )(\lambda) \right | \le \left(b + \dfrac{1+4\beta}{1-4\beta} \left ( \int _{\mathbb T} \dfrac{1 - |\lambda _0|^2}{|1 - \bar \lambda _0 z |^2} |g|^qd\mu \right )^{\frac{1}{q}} \right ) \|f\|_{L^p(\mu )}
\]
for all $\lambda\in B (\lambda _0, \beta \delta ) \setminus E_\delta^f.$
\end{prop}

\begin{proof} Let $\nu = \mu | _{\mathbb D}.$ We now apply Lemma \ref{CauchyTLemma} for $p, ~ q,~f,~ g, $ and $a=\frac{1}{2^{\frac{1}{p}}}b.$ From Lemma 3.5 in \cite{acy18}, there exists $E$ with $\gamma(E) = 0$ such that for $e^{i\theta} \in \mathbb T\setminus E$, $|g|^qd|\nu|$ satisfies \eqref{zerodensity}. There exist $0 < \delta_a < \frac{1}{4},$ $E_{\delta}^f \subset B(e^{i\theta}, \delta),$ and $\epsilon (\delta ) > 0,$ where $0 < \delta < \delta_a ,$ such that $\lim_{\delta\rightarrow 0}\epsilon (\delta ) = 0,$ $\gamma(E_\delta ^f) < \epsilon (\delta ) \delta ,$ and for $\lambda _0 \in (\partial B (e^{i\theta}, \frac{\delta}{2} )) \cap \Gamma _{\frac{1}{4}}(e^{i\theta}),$
\[
\ \left | \mathcal C \left (\dfrac{(1 - \bar \lambda _0 z)^{\frac{2}{p}}}{ (1-|\lambda _0|^2)^{\frac{1}{p}}}fg\nu \right )(\lambda ) - \mathcal C \left (\dfrac{(1 - \bar \lambda _0 z)^{\frac{2}{p}}}{ (1-|\lambda _0|^2)^{\frac{1}{p}}}fg\nu \right )(\frac{1}{\bar{\lambda} _0})\right | \le b  \|f\|_{L^p(\mu )}
 \]
for all $\lambda\in B (e^{i\theta }, \delta ) \setminus E_\delta^f.$
  \[
\ \mathcal C \left (\dfrac{(1 - \bar \lambda _0 z)^{\frac{2}{p}}}{ (1-|\lambda _0|^2)^{\frac{1}{p}}}fg\mu \right )(\frac{1}{\bar{\lambda} _0}) = 0
\]
since $fg\mu \perp Rat(K).$ Therefore, for all $\lambda\in B (\lambda _0, \beta \delta ) \setminus E_\delta^f,$ we get
 \[
 \ \begin{aligned}
 \ & \left  |\mathcal C \left (\dfrac{(1 - \bar \lambda _0 z)^{\frac{2}{p}}}{ (1-|\lambda _0|^2)^{\frac{1}{p}}}fg\mu \right )(\lambda) \right | \\
 \ \le & \left  |\mathcal C \left (\dfrac{(1 - \bar \lambda _0 z)^{\frac{2}{p}}}{ (1-|\lambda _0|^2)^{\frac{1}{p}}}fg\nu \right )(\lambda) - \mathcal C \left (\dfrac{(1 - \bar \lambda _0 z)^{\frac{2}{p}}}{ (1-|\lambda _0|^2)^{\frac{1}{p}}}fg\nu \right )(\frac{1}{\bar{\lambda} _0})\right | \\
 \ & + \left  |\int_{\mathbb T} \left ( \dfrac{1}{z - \lambda} - \dfrac{1}{z - \frac{1}{\bar{\lambda} _0}} \right )\dfrac{(1 - \bar \lambda _0 z)^{\frac{2}{p}}}{ (1-|\lambda _0|^2)^{\frac{1}{p}}}fg\mu\right | \\
 \ \le & b \|f\|_{L^p(\mu )} + \int_{\mathbb T} \dfrac{|1 - \lambda \bar{\lambda} _0| }{|z - \lambda |} \dfrac{ (1 - |\lambda _0|^2)^{-\frac{1}{p}} }{ |1 - \bar \lambda _0 z|^{1 - \frac{2}{p}}} |fg| d \mu \\
 \ \le & b \|f\|_{L^p(\mu )} + \dfrac{1+4\beta}{1-4\beta} \int_{\mathbb T} \dfrac{ (1 - |\lambda _0|^2)^{\frac{1}{q}}}{ |1 - \bar \lambda _0 z|^{\frac{2}{q}}} |fg| d \mu 
 \ \end{aligned}
 \]
 where the last step follows from
 \[
 \ \begin{aligned}
 \ \dfrac{|1 - \lambda \bar{\lambda} _0| }{|z - \lambda |} \le & \dfrac{1 - |\lambda _0|^2 + |\lambda _0||\lambda - \lambda _0| }{|z - \lambda _0 | - |\lambda - \lambda _0|} \\
\  \le & \dfrac{(1+4\beta)(1 - |\lambda _0|^2)}{|z - \lambda _0 | - 4\beta (1 - |\lambda _0|)} \\
\ \le &\dfrac{(1+4\beta)(1 - |\lambda _0|^2)}{(1 - 4\beta )|1 - \bar\lambda _0 z|}
\ \end{aligned} 
\]
for $z\in\mathbb T.$ The proposition now follows from Holder's inequality.
\end{proof}
\medskip

Let $R = \{z :~ |Re(z)| < \frac{1}{2}\text{ and }|Im(z)| < \frac{1}{2}\}$ and $Q = \bar{\mathbb D}\setminus R.$ For a bounded Borel set
$E\subset \mathbb C$ and $1\le p \le \infty,$ $L^p(E)$ denotes the $L^p$ space with respect to the area measure $dA$ restricted to $E.$  The following Lemma is a simple application of Thomson's coloring scheme.

\begin{lem} \label{lemmaDSet}
There is an absolute constant $\epsilon _1 > 0$ with the
following property. If $\gamma (\mathbb D \setminus K) < \epsilon_1,$ then
\[
\ |f(\lambda ) | \le \|f\|_{L^\infty (Q\cap K)}
\]
for $\lambda \in R$ and $f \in A(\mathbb D),$ the uniform closure of $\mathcal P$ in $C(\bar {\mathbb D}).$
\end{lem}

\begin{proof} We use Thomson's coloring scheme that is described at the beginning of section 2 of \cite{y19}. Let $\epsilon_1$ be chosen as in Lemma 2 of \cite{y19}. By our assumption $\gamma (\mathbb D \setminus K) < \epsilon_1$ and Lemma 2 of \cite{y19}, we conclude that Case II on Page 225 of \cite{y19} holds, that is, $scheme(Q, \epsilon, m, \gamma_n, \Gamma_n, n \ge m)$ ($\epsilon < 10^{-3}$) does not terminate. In this case, one has a sequence of heavy $\epsilon$ barriers inside $Q,$ that is, $\{\gamma_n\}_{n\ge m}$ and $\{\Gamma_n\}_{n\ge m}$ are infinite.

Let $f\in A(\mathbb D),$ by the maximal modulus principle, we can find $z_n\in\gamma_n$ such that $|f(\lambda )| \le |f(z_n)|$ for $\lambda \in R.$ By the definition of $\gamma_n,$ we can find a heavy $\epsilon$ square $S_n$ with $z_n\in S_n\cap\gamma_n.$ Since $\gamma(Int(S_n)\setminus K) \le \epsilon d_{S_n}$ (see (2.2) in \cite{y19}), we must have $Area(S_n\cap K) > 0$. We can choose $w_n\in S_n\cap K$ with $|f(w_n)| = \|f\|_{L^\infty (S_n\cap K)}.$ $\frac{f(w)-f(z_n)}{w-z_n}$ is analytic in $\mathbb D,$ therefore, by the maximal modulus principle again, we get
 \[
 \ \left | \dfrac{f(w_n)-f(z_n)}{w_n-z_n} \right | \le \sup_{w \in \gamma_{n+1}} \left | \dfrac{f(w)-f(z_n)}{w-z_n} \right | \le \dfrac{2\|f\|_{L^\infty (\mathbb D)}}{dist (z_n,\gamma_{n+1})}.
 \]
By Lemma 2.1 in \cite{thomson} (there is a buffer zone of yellow squares between $\gamma_n$ and $\gamma_{n+1}$), we have $dist (z_n,\gamma_{n+1})\ge n^2 2^{-n}$.  
Therefore,
 \[
 \ \begin{aligned}
 \ |f(\lambda )| \le & |f(z_n)| \le |f(w_n)| + \dfrac{2|z_n-w_n|\|f\|_{L^\infty (\mathbb D)}}{dist (z_n,\gamma_{n+1})} \\ 
 \ \le &\|f\|_{L^\infty(Q\cap K)} + \dfrac{2\sqrt 2 2^{-n}\|f\|_{L^\infty (\mathbb D)}}{n^2 2^{-n}}
\ \end{aligned}
 \]
 for $\lambda \in R.$
The lemma follows by taking $n\rightarrow \infty .$
\end{proof}
\medskip

\begin{cor} \label{CorollaryDSet}
There is an absolute constant $\epsilon _1 > 0$ with the
following property. If $\lambda _0 \in \mathbb C, ~ \delta > 0,$ and $\gamma (B(\lambda _0 , \delta) \setminus K) < \epsilon_1\delta ,$ then
\[
\ |f(\lambda ) | \le \|f\|_{L^\infty (B(\lambda _0 , \delta)\cap K)}
\]
for $\lambda \in B(\lambda _0 , \frac{\delta}{2})$ and $f \in A(B(\lambda _0 , \delta)),$ the uniform closure of $\mathcal P$ in $C(\overline {B(\lambda _0 , \delta)}).$
\end{cor}
\medskip

\begin{proof} ({\bf Main Theorem}): From Lemma VII.1.7 in \cite{conway}, we find a function $G \in R^t(K,\mu )^\perp$ such that $G(z) \ne 0$ for $\mu$-almost every $z.$ There exists $Z_1 \subset \mathbb T$ with $m(Z_1) = 0$ such that $G(e^{i\theta})h(e^{i\theta}) \ne 0$ for $e^{i\theta}\in \partial _{so} \Omega \cap \mathcal N(h)\setminus Z_1$, where $\mathcal N(h) = \{e^{i\theta}:~ h(e^{i\theta}) > 0\}$.

By Theorem 3.6 (Plemelj’s Formula for an arbitrary measure) in \cite{acy18}, for $e^{i\theta}\in \partial _{so} \Omega \cap \mathcal N(h) \setminus Z_1 \setminus Z_2$ with  $m(Z_2) = 0$, $\Gamma ^{r_0}_\frac{1}{2}(e^{i\theta})  \subset \Omega ,$ and $b > 0$, there exist $0 < \delta_a^0 < 1-max(\frac{3}{4},  r_0),$ $E_{\delta} \subset B(e^{i\theta}, \delta)$, and $\epsilon^0 (\delta ) > 0,$ where $0 < \delta < \delta_a^0,$ such that $\lim_{\delta\rightarrow 0}\epsilon ^0(\delta ) = 0,$ $\gamma(E_\delta ) < \epsilon ^0(\delta ) \delta ,$ 
\[
\ \left  |\mathcal C(G\mu ) (\lambda) - \mathcal C(G\mu ) (e^{i\theta}) - \frac{1}{2}e^{-i\theta}G(e^{i\theta})h(e^{i\theta}) \right | \le \frac{b}{2}
\]
and
\[
\ \left  |\mathcal C(G\mu ) (\frac{1}{\bar\lambda}) - \mathcal C(G\mu ) (e^{i\theta}) + \frac{1}{2}e^{-i\theta}G(e^{i\theta})h(e^{i\theta}) \right | \le \frac{b}{2},
\]
hence,
\begin{eqnarray}\label{MTheoremEq1}
\ \left  |\mathcal C(G\mu ) (\lambda) - e^{-i\theta}G(e^{i\theta})h(e^{i\theta}) \right | \le b
\end{eqnarray}
since $\mathcal C(G\mu ) (\frac{1}{\bar\lambda}) = 0$ for all $\lambda\in B (e^{i\theta}, \delta ) \cap \Gamma _\frac{1}{2}(e^{i\theta})\setminus E_\delta.$

By Proposition \ref{MProposition2} for $p = t$, $q = t'$, $f\in Rat(K) \subset C(\overline{\mathbb D})$, and $g = G$, for $e^{i\theta}\in \partial _{so} \Omega \cap \mathcal N(h) \setminus Z_1 \setminus Z_3$ with  $m(Z_3) = 0$, $\Gamma ^{r_1}_\frac{1}{2}(e^{i\theta})  \subset \Omega ,$  $0 < \beta < \frac{1}{16}$, and  $b > 0$,  there exist $0 < \delta_a^1 < 1-max(\frac{3}{4},  r_1),$ $E_{\delta}^f \subset B(e^{i\theta}, \delta),$ and $\epsilon^1 (\delta ) > 0,$ where $0 < \delta < \delta_a^1,$ such that $\lim_{\delta\rightarrow 0}\epsilon ^1(\delta ) = 0,$ $\gamma(E_\delta ^f) < \epsilon ^1(\delta ) \delta ,$  
\begin{eqnarray}\label{MTheoremEq2}
\ \begin{aligned}
\ &\left  |\mathcal C \left (\dfrac{(1 - \bar \lambda _0 z)^{\frac{2}{t}}}{ (1-|\lambda _0|^2)^{\frac{1}{t}}}fG\mu \right )(\lambda) \right | \\
\ \le &\left(b + \dfrac{1+4\beta}{1-4\beta} \left ( \int _{\mathbb T} \dfrac{1 - |\lambda _0|^2}{|1 - \bar \lambda _0 z |^2} |G|^{t'}d\mu \right )^{\frac{1}{t'}} \right ) \|f\|_{L^t(\mu )}
\ \end{aligned}
\end{eqnarray}
for $\lambda _0 \in \partial B (e^{i\theta}, \frac{\delta}{2} ) \cap \Gamma _{\frac{1}{4}}(e^{i\theta})$ and all $\lambda\in B (\lambda _0, \beta \delta ) \setminus E_\delta^f.$ 

Set $Z = Z_1\cup Z_2\cup Z_3$. For $e^{i\theta}\in \partial _{so} \Omega \cap \mathcal N(h) \setminus Z$, set $\delta_a = min(\delta_a^0,\delta_a^1)$ and $\epsilon (\delta ) = min(\epsilon^0 (\delta ),\epsilon^1 (\delta ))$. Then for $e^{i\theta}\in \partial _{so} \Omega \cap \mathcal N(h) \setminus Z$ and $0 < \delta < \delta_a,$ \eqref{MTheoremEq1} and \eqref{MTheoremEq2} hold. 
From semi-additivity of Theorem \ref{TTolsa} (2), we get
 \[
 \ \gamma (E_\delta \cup E_\delta^f) \le A_T(\gamma (E_\delta ) + \gamma ( E_\delta^f)) \le 2A_T \epsilon (\delta ) \delta .
 \] 
Let $\delta $ be small enough so that $\epsilon (\delta ) < \frac{\beta}{2A_T}\epsilon _1,$ where $\epsilon _1$ is as in Corollary \ref{CorollaryDSet}. 
For $\lambda _0 \in \partial B (e^{i\theta}, \frac{\delta }{2}) \cap \Gamma _{\frac{1}{4}}(e^{i\theta})$ and all $\lambda\in B (\lambda _0, \beta \delta ) \setminus (E_{\delta } \cup E_{\delta }^f),$ it is clear that
\[
\ f(\lambda ) \mathcal C(G\mu ) (\lambda ) = \int \dfrac{f(z)}{z - \lambda }G(z) d\mu(z) = \mathcal C(fG\mu ) (\lambda )
\]
since $\frac{f(z) - f(\lambda}{z - \lambda} \in R^t(K, \mu)$. Together with \eqref{MTheoremEq1} and \eqref{MTheoremEq2}, we have the following calculation:
 \[
 \ | 1 - \bar \lambda _0 \lambda | \ge 1- |\bar \lambda _0|^2 - |\lambda - \lambda _0||\bar \lambda _0| \ge 1- |\bar \lambda _0|^2 - \beta \delta |\lambda _0|
 \]
and 
 \[
 \ \begin{aligned}
 \ (1-|\lambda _0|^2)^{\frac{1}{t}} |f(\lambda ) | \le & \dfrac{| (1 - \bar \lambda _0 \lambda )^{\frac{2}{t}}(1-|\lambda _0|^2)^{-\frac{1}{t}}f(\lambda ) |}{(1  - \beta \frac{\delta  |\lambda _0|}{1-|\lambda _0|^2})^{\frac{2}{t}}} \\
 \ = & \dfrac{1}{(1  - \beta \frac{\delta  |\lambda _0|}{1-|\lambda _0|^2})^{\frac{2}{t}}} \left |\dfrac{  \mathcal C \left (\dfrac{(1 - \bar \lambda _0 z)^{\frac{2}{t}}}{ (1-|\lambda _0|^2)^{\frac{1}{t}}}fG\mu \right )(\lambda) }{\mathcal C(G\mu ) (\lambda)} \right | \\
\ \le & \dfrac{ b + \frac{1+4\beta}{1-4\beta} \left ( \int _{\mathbb T} \frac{1 - |\lambda _0|^2}{|1 - \bar \lambda _0 z |^2} |G|^{t'}d\mu \right )^{\frac{1}{t'}} }{(1-4\beta)^{\frac{2}{t}}(|G(e^{i\theta})|h(e^{i\theta}) - b)} \|f\|_{L^t(\mu )}.
\ \end{aligned}
\] 
Since $\gamma (E_{\delta }\cup E_{\delta }^f) < \epsilon _1 \beta \delta ,$ from Corollary \ref{CorollaryDSet}, we conclude
 \[
 \ M_{\lambda _0} \le \sup _{\underset{\|f\|_{L^t(\mu )} = 1}{f\in Rat(K)}}|f(\lambda _0)| \le \sup _{\underset{\|f\|_{L^t(\mu )} = 1}{f\in Rat(K)}}\|f\|_{L^\infty (B (\lambda _0, \beta \delta ) \setminus (E_{\delta } \cup E_{\delta }^f)) }  
 \]
for $\lambda _0 \in \partial B (e^{i\theta}, \frac{\delta}{2} ) \cap \Gamma _{\frac{1}{4}}(e^{i\theta}).$  Hence,
 \[
 \ \underset{\Gamma _{\frac{1}{4}}(e^{i\theta})\ni \lambda _0 \rightarrow e^{i\theta}}{\overline\lim} (1-|\lambda _0|^2)^{\frac{1}{t}} M_{\lambda _0} \le \dfrac{ b + \frac{1+4\beta}{1-4\beta} |G(e^{i\theta})|(h(e^{i\theta}))^{\frac{1}{t'}} }{(1-4\beta)^{\frac{2}{t}}(|G(e^{i\theta})|h(e^{i\theta}) - b)}
 \]
since $\frac{1 - |\lambda _0|^2}{|1 - \bar \lambda _0 z |^2}$ is the Poisson kernel. Taking $b\rightarrow 0$ and $\beta\rightarrow 0, $ we get
 \[
 \ \lim_{\Gamma _{\frac{1}{4}}(e^{i\theta}) \ni \lambda \rightarrow e^{i\theta}} (1 - |\lambda |^2)^{\frac{1}{t}} M_\lambda \le \dfrac{1}{h(e^{i\theta})^{\frac{1}{t}}}.
 \]
The reverse inequality is from Lemma 1 in \cite{kt77} (applying the lemma to testing function $(1 - \bar \lambda _0 z )^{-\frac{2}{t}}$). This completes the proof.
\end{proof}
\bigskip

\centerline{\bf Acknowledgment}

The author would like to thank the referee for providing helpful comments.

\end{document}